\documentclass[12pt,twoside]{amsart}
\usepackage{amsmath, amsthm, amscd, amsfonts, amssymb, graphicx}
\usepackage{enumerate}
\usepackage[colorlinks=true,
linkcolor=blue,
urlcolor=cyan,
citecolor=red]{hyperref}
\usepackage{mathrsfs}
\newtheorem{theorem}{Theorem}[section]
\newtheorem{lemma}{Lemma}[section]
\newtheorem{remark}{Remark}[section]

\newtheorem{corollary}{Corollary}[section]

\newtheorem{proposition}{Proposition}[section]
\numberwithin{equation}{section}

\begin{document}
	
\title{New estimates for the numerical radius}
\author{Hamid Reza Moradi and Mohammad Sababheh}
\subjclass[2010]{Primary 47A12, 47A30, Secondary 15A60, 47B15.}
\keywords{Numerical radius, operator norm, accretive-dissipative operator.}

\begin{abstract}
In this article, we present new inequalities for the numerical radius of the sum of two Hilbert space operators. These new inequalities will enable us to obtain many generalizations and refinements of some well known inequalities, including multiplicative behavior of the numerical radius and norm bounds.

Among many other applications, it is shown that if $T$ is  accretive-dissipative, then
\[\frac{1}{\sqrt{2}}\left\| T \right\|\le \omega \left( T \right),\]
where $\omega \left( \cdot \right)$ and $\left\| \cdot \right\|$ denote the numerical radius and the usual operator norm, respectively. This inequality provides a considerable refinement of the well known inequality $\frac{1}{2}\|T\|\leq \omega(T).$ 
\end{abstract}
\maketitle
\pagestyle{myheadings}
\markboth{\centerline {}}
{\centerline {}}
\bigskip
\bigskip

\section{Introduction}
Let $\mathcal{H}$ be a complex Hilbert space, endowed with the inner product $\left<\cdot,\cdot\right>:\mathcal{H}\times \mathcal{H}\to\mathbb{C}.$ The $C^*$-algebra of all bounded linear operators on $\mathcal{H}$ will be denoted by $\mathbb{B}(\mathcal{H}).$ An operator $A\in\mathbb{B}(\mathcal{H})$ is said to be positive (denoted by $A>0$) if $\left<Ax,x\right> > 0$ for all non-zero vectors $x\in\mathcal{H}$, and self-adjoint if $A^*=A$, where $A^*$ is the adjoint of $A$. For $A\in\mathbb{B}(\mathcal{H})$, the Cartesian decomposition of $A$ is $A=\Re A+i \Im A$, where 
$$\Re A=\frac{A+A^*}{2}\;{\text{and}}\;\Im A=\frac{A-A^*}{2i},$$ are the real and imaginary parts of $A$, respectively. It is clear that both $\Re A$ and $\Im A$ are self-adjoint. If $\Re A>0$, the operator $A$ will be called accretive, and if both $\Re A$ and $\Im A $ are positive, $A$ will be called accretive-dissipative.

Among the most interesting scalar quantities associated with an operator $A\in\mathbb{B}(\mathcal{H})$ are the the usual operator norm and the numerical radius, defined respectively by
$$\|A\|=\sup_{\|x\|=1}\|Ax\|\;{\text{and}}\;\omega(A)=\sup_{\|x\|=1}\left|\left<Ax,x\right>\right|.$$
It is well known that $\|A\|=\sup_{\|x\|=\|y\|=1}\left|\left<Ax,y\right>\right|$. The operator norm and the numerical radius are always comparable by the inequality
\begin{align}\label{eq_intro_1}
\frac{1}{2}\|A\|\leq \omega(A)\leq \|A\|, A\in\mathbb{B}(\mathcal{H}).
\end{align}
When $A\in\mathbb{B}(\mathcal{H})$ is normal (i.e., $AA^*=A^*A$), we have $\|A\|=\omega(A)$.

The significance of \eqref{eq_intro_1} is having upper and lower bounds of $\omega(A)$; a quantity that is not easy to calculate. Due to the importance and applicability of the quantity $\omega(A)$, interest has grown in having better bounds of $\omega(A)$ than the bounds in \eqref{eq_intro_1}. 

The main goal of this article is to present new inequalities for the numerical radius. More precisely, we study inequalities for quantities of the forms $$\omega(A+B)\;{\text{and}}\;\omega(A+iB),\;A,B\in\mathbb{B}(\mathcal{H}).$$

Our approach will be based mainly on the scalar inequality
\begin{equation}\label{1}
\left| a+b \right|\le \sqrt{2}\left| a+ib \right|,
\end{equation}
valid for $a,b\in\mathbb{R}.$ Using this inequality is a new approach to tickle numerical radius inequalities. This approach will enable us to obtain new bounds, some of which are generalizations of certain known bounds. For example, we will refine Kittaneh inequality  \cite{1}, 
\begin{equation}\label{kitt_ineq_intro}
\frac{1}{4}\left\| {{\left| A \right|}^{2}}+{{\left| {{A}^{*}} \right|}^{2}} \right\|\le {{\omega }^{2}}\left( A \right)\le \frac{1}{2}\left\| {{\left| A \right|}^{2}}+{{\left| {{A}^{*}} \right|}^{2}} \right\|;
\end{equation}
noting that the inequalities \eqref{kitt_ineq_intro} refine the inequalities \eqref{eq_intro_1}.

In fact, our approach will not enable us to refine these inequalities only, it will present a new proof and generalization of \eqref{kitt_ineq_intro}.

Moreover, using this approach, we will show that
$$\frac{1}{\sqrt{2}}\|A\|\leq \omega(A)$$ for the accretive-dissipative operator $A$. This inequality presents a considerable improvement of the first inequality in \eqref{eq_intro_1}. As a result, we will be able to introduce a better bound for the sub-multiplicative behavior of the numerical radius, when dealing with accretive-dissipative operators. More precisely, we show that $\omega(ST)\leq 2\omega(S)\omega(T)$ when both $S$ and $T$ are accretive-dissipative. See Corollary \ref{cor_submult} below for further discussion.

Other results including  reverses of \eqref{kitt_ineq_intro}  and a refinement of the triangle inequality will be shown too.

In our proofs, we will need to recall the following inequalities.
 \begin{lemma}\label{lemma1}
\cite[pp. 75--76]{2} Let $A\in\mathbb{B}(\mathcal{H})$ and let $x\in\mathcal{H}$. Then 
$$|\left<Ax,x\right>|\leq\sqrt{\left<|A|x,x\right>\left<|A^*|x,x\right>}.$$
\end{lemma}

\begin{lemma}\label{10}
Let $A\in \mathbb{B}\left( \mathcal{H} \right)$ be  self-adjoint. Then for any unit vector $x\in \mathcal{H}$,
\[{{\left\langle Ax,x \right\rangle }^{2}}\le \left\langle {{A}^{2}}x,x \right\rangle .\]
\end{lemma}

Also, we will need the following result \cite[Proposition 3.8]{lin}.
\begin{lemma}\label{lem_lin}
Let $A,B\in\mathbb{B}(\mathcal{H})$ be positive. Then
$$\|A+iB\|\leq \|A+B\|.$$
\end{lemma}
In the following proposition, we restate  \eqref{kitt_ineq_intro} in terms of the Cartesian decomposition. 
\begin{proposition}\label{5}
Let $A,B\in \mathbb{B}\left( \mathcal{H} \right)$ be self-adjoint. Then
$$\frac{1}{2}\|A^2+B^2\|\leq \omega^2(A+iB)\leq \|A^2+B^2\|.$$
\end{proposition}

\section{New Results}

We begin our main results with the the numerical radius version of the inequality \eqref{1}, which can be stated as follows.
\begin{theorem}\label{2}
Let $A,B\in \mathbb{B}\left( \mathcal{H} \right)$. Then
\[\omega \left( A+B \right)\le \frac{1}{\sqrt{2}}\omega \left( \left( \left| A \right|+\left| B \right| \right)+i\left( \left| {{A}^{*}} \right|+\left| {{B}^{*}} \right| \right) \right).\]
Further, this inequality is sharp, in the sense that the factor $\frac{1}{\sqrt{2}}$ cannot be replaced by a smaller number.
\end{theorem}
\begin{proof}
Let $x\in \mathcal{H}$ be a unit vector. Then
\[\begin{aligned}
   \left| \left\langle \left( A+B \right)x,x \right\rangle  \right|&\le \left| \left\langle Ax,x \right\rangle  \right|+\left| \left\langle Bx,x \right\rangle  \right| \\ 
 & \le \left| \sqrt{\left\langle \left| A \right|x,x \right\rangle \left\langle \left| {{A}^{*}} \right|x,x \right\rangle }+\sqrt{\left\langle \left| B \right|x,x \right\rangle \left\langle \left| {{B}^{*}} \right|x,x \right\rangle } \right| \\ 
 & \le \left| \frac{1}{2}\left( \left\langle \left| A \right|x,x \right\rangle +\left\langle \left| {{A}^{*}} \right|x,x \right\rangle  \right)+\frac{1}{2}\left( \left\langle \left| B \right|x,x \right\rangle +\left\langle \left| {{B}^{*}} \right|x,x \right\rangle  \right) \right| \\ 
 & =\left| \left\langle \left( \frac{\left| A \right|+\left| B \right|}{2} \right)x,x \right\rangle +\left\langle \left( \frac{\left| {{A}^{*}} \right|+\left| {{B}^{*}} \right|}{2} \right)x,x \right\rangle  \right| \\ 
 & \le \sqrt{2}\left| \left\langle \left( \frac{\left| A \right|+\left| B \right|}{2} \right)x,x \right\rangle +i\left\langle \left( \frac{\left| {{A}^{*}} \right|+\left| {{B}^{*}} \right|}{2} \right)x,x \right\rangle  \right| \\ 
 & =\sqrt{2}\left| \left\langle \left( \frac{\left| A \right|+\left| B \right|}{2}+i\frac{\left| {{A}^{*}} \right|+\left| {{B}^{*}} \right|}{2} \right)x,x \right\rangle  \right|,
\end{aligned}\]
where the first inequality follows from the triangle inequality, the second inequality is obtained by Lemma  \ref{lemma1}, the third inequality is obtained by the arithmetic-geometric mean inequality and the forth inequality, follows from \eqref{1}.
Therefore, we have shown that for any unit vector $x\in\mathcal{H}$,
\[\left| \left\langle \left( A+B \right)x,x \right\rangle  \right|\le \frac{1}{\sqrt{2}}\left| \left\langle \left( \left( \left| A \right|+\left| B \right| \right)+i\left( \left| {{A}^{*}} \right|+\left| {{B}^{*}} \right| \right) \right)x,x \right\rangle  \right|.\]
Now, by taking supremum over all unit vector $x\in \mathcal{H}$, we get
\[\omega \left( A+B \right)\le \frac{1}{\sqrt{2}}\omega \left( \left( \left| A \right|+\left| B \right| \right)+i\left( \left| {{A}^{*}} \right|+\left| {{B}^{*}} \right| \right) \right),\]
	which completes the proof of the first assertion of the Theorem. To show that the factor $\frac{1}{\sqrt{2}}$ is best possible, let $B=0$ and assume that $A$ is positive. Direct calculations show that the inequality is sharp, which completes the proof.
\end{proof}

The following result shows how Theorem \ref{2} refines \eqref{kitt_ineq_intro}.
\begin{corollary}
Let $A\in \mathbb{B}\left( \mathcal{H} \right)$. Then
\[{{\omega }^{2}}\left( A \right)\le \frac{1}{2}{{\omega }^{2}}\left( \left| A \right|+i\left| {{A}^{*}} \right| \right)\le \frac{1}{2}\left\| {{\left| A \right|}^{2}}+{{\left| {{A}^{*}} \right|}^{2}} \right\|.\]
\end{corollary}
\begin{proof}
The first inequality follows from Theorem \ref{2} by taking $B=0$ and the second inequality follows  from Proposition \ref{5}.
This completes the proof.
\end{proof}

Using the same method presented in Theorem \ref{2}, we can obtain the following result; a different form of Theorem \ref{2}.
\begin{theorem}
Let $A,B\in \mathbb{B}\left( \mathcal{H} \right)$. Then
	\[\omega \left( A+B \right)\le \frac{1}{\sqrt{2}}\omega \left( \left( \left| A \right|+\left| {{A}^{*}} \right| \right)+i\left( \left| B \right|+\left| {{B}^{*}} \right| \right) \right).\]
	Further, the factor $\frac{1}{\sqrt{2}}$ is best possible.
\end{theorem}
The next Corollary follows from Theorem \ref{2} and by taking into account that the sum of two normal operators, need not necessarily a normal operator.
\begin{corollary}\label{3}
Let $A,B\in \mathbb{B}\left( \mathcal{H} \right)$ be two normal operators. Then
\[\omega \left( A+B \right)\le \sqrt{2}\omega \left( \left| A \right|+i\left| B \right| \right).\]
In particular, if $T=A+iB$ is accretive-dissipative, then
\begin{equation}\label{4}
\omega(T)\geq\frac{1}{\sqrt{2}}\|A+B\|.
\end{equation}
\end{corollary}

\begin{remark}
In \cite{bhatia} (See also \cite[Lemma 3.5]{lin}), it is shown that if $A,B\in\mathbb{B}(\mathcal{H})$ are self-adjoint, then
\begin{align}\label{lin_lem_eq}
\|A+B\|_u\leq \sqrt{2}\|A+iB\|_u,
\end{align}
for any unitarily invariant norm $\|\cdot\|_u$. It is implicitly understood that $\|\cdot\|_u$ is defined on an ideal in $\mathbb{B}(\mathcal{H})$, and it is implicitly understood that $T$ is in that ideal, when we speak of $\|T\|_u$.  We notice, first, that Corollary \ref{3} provides the numerical radius version of \eqref{lin_lem_eq}, in which $A,B$ are normal; a wider class than positive matrices. Further, Corollary \ref{3} provides a refinement of \eqref{lin_lem_eq} in case of the usual operator norm since
$$\|A+B\| \leq \sqrt{2}\omega(A+iB)\leq\sqrt{2}\|A+iB\|.$$

\end{remark}

The next result provides a considerable improvement of the first inequality in \eqref{eq_intro_1}, for accretive-dissipative operators.
\begin{theorem}\label{6}
Let $T\in \mathbb{B}\left( \mathcal{H} \right)$ be  accretive-dissipative. Then
\[\frac{1}{\sqrt{2}}\left\| T \right\|\le \omega \left( T \right).\]
\end{theorem}
\begin{proof}
Let $T=A+iB$ be the Cartesian decomposition of $T$, in which both $A,B$ are positive. Then Corollary \ref{3} together with Lemma \ref{lem_lin} imply
$$\|T\|=\|A+iB\|\geq \|A+B\|\leq \sqrt{2}\omega(T).$$ This completes the proof.
\end{proof}

From \eqref{eq_intro_1} and the fact that the operator norm is sub-multiplicative, we obtain the well known inequality
$$\omega(AB)\leq\|AB\|\leq \|A\|\;\|B\|\leq 4\omega(A)\omega(B).$$ It is well established that the factor 4 cannot be replaced by a smaller factor in general. However, when $A$ or $B$ is normal, we obtain the better bound $\omega(AB)\leq 2\omega(A)\omega(B)$, and it is even better when both are normal as we have $\omega(AB)\leq \omega(A)\omega(B).$ In the following result, we present a new bound for accretive-dissipative operators, which is better than the bound $\omega(AB)\leq 4\omega(A)\omega(B).$ We refer the reader to \cite{gusta} for detailed study of this problem.
\begin{corollary}\label{cor_submult}
Let $S,T\in \mathbb{B}\left( \mathcal{H} \right)$ be two accretive-dissipative operators. Then
\[\omega \left( ST \right)\le 2\omega \left( S \right)\omega \left( T \right).\]
If either $T$ or $S$ is accretive-dissipative, then
$$\omega(ST)\leq 2\sqrt{2}\omega(S)\omega(T).$$
\end{corollary}
\begin{proof}
Noting submultiplicativity of the operator norm and Theorem \ref{6}, we have
\[\begin{aligned}
   \omega \left( ST \right)&\le \left\| ST \right\| 
 & \le \left\| S \right\|\left\| T \right\|
 & \le 2\omega \left( S \right)\omega \left( T \right),
\end{aligned}\]
which completes the proof of the first inequality. The second inequality follows similarly.
\end{proof}

It is interesting that the approach we follow in this paper allows us to obtain reversed inequalities, as well. In \cite{1}, it is shown that
$$\frac{1}{4}\left\| \;|A|^2+|A^*|^2\; \right\|\le {{w}^{2}}(A).$$ In the following, we present a refinement of this inequality using our approach.
\begin{corollary}
Let $T\in \mathbb{B}\left( \mathcal{H} \right)$ have the Cartesian decomposition $T=A+iB.$ Then
\[\begin{aligned}
   \frac{1}{4}\left\| {{\left| T \right|}^{2}}+{{\left| {{T}^{*}} \right|}^{2}} \right\|&\le \frac{\sqrt{2}}{2}\omega \left( A^2+iB^2 \right) 
 & \le {{\omega }^{2}}\left( T \right).  
\end{aligned}\]
\end{corollary}
\begin{proof}
In Corollary \ref{3}, replace $A$ and $B$ by ${{A}^{2}}$ and ${{B}^{2}}$. This implies
\[\begin{aligned}
   \left\| {{A}^{2}}+{{B}^{2}} \right\|&\le \sqrt{2}\omega \left( {{A}^{2}}+i{{B}^{2}} \right) \\ 
 & \le \sqrt{2\left\| {{A}^{4}}+{{B}^{4}} \right\|} \quad \text{(by Proposition \ref{5})}\\ 
 & \le \sqrt{2\left( {{\left\| A \right\|}^{4}}+{{\left\| B \right\|}^{4}} \right)},  
\end{aligned}\]
where the last inequality follows by the triangle inequality for the usual operator norm.
If $T=A+iB$ is the Cartesian decomposition of the operator $T$, then
\[\frac{1}{2}\left\| {{\left| T \right|}^{2}}+{{\left| {{T}^{*}} \right|}^{2}} \right\|=\left\| {{A}^{2}}+{{B}^{2}} \right\|\text{ and }\sqrt{2\left( {{\left\| A \right\|}^{4}}+{{\left\| B \right\|}^{4}} \right)}\le 2{{\omega }^{2}}\left( T \right).\]
Therefore
\[\frac{1}{4}\left\| {{\left| T \right|}^{2}}+{{\left| {{T}^{*}} \right|}^{2}} \right\|\le \frac{\sqrt{2}}{2}\omega \left( {{A}^{2}}+i{{B}^{2}} \right)\le {{\omega }^{2}}\left( T \right),\]
which is equivalent to the desired result.
\end{proof}

On the other hand, manipulating Proposition \ref{5} implies the following refinement of the triangle inequality $$\|A+B\|\leq \|A\|+\|B\|.$$ The connection of this result to our analysis is the refining term $\omega(A+iB)$.

\begin{theorem}\label{0}
Let $A,B\in \mathcal{B}\left( \mathcal{H} \right)$ be two self-adjoint operators. Then
\[\left\| A+B \right\|\le \sqrt{{{\omega }^{2}}\left( A+iB \right)+2\left\| A \right\|\left\| B \right\|}\le \left\| A \right\|+\left\| B \right\|.\]
\end{theorem}
\begin{proof}
Let $x\in\mathcal{H}$ be a unit vector. Then
\[\begin{aligned}
   {{\left| \left\langle \left( A+B \right)x,x \right\rangle  \right|}^{2}}&\le {{\left( \left| \left\langle Ax,x \right\rangle  \right|+\left| \left\langle Bx,x \right\rangle  \right| \right)}^{2}} \\ 
 & ={{\left| \left\langle Ax,x \right\rangle  \right|}^{2}}+{{\left| \left\langle Bx,x \right\rangle  \right|}^{2}}+2\left| \left\langle Ax,x \right\rangle  \right|\left| \left\langle Bx,x \right\rangle  \right| \\ 
 & ={{\left| \left\langle \left( A+iB \right)x,x \right\rangle  \right|}^{2}}+2\left| \left\langle Ax,x \right\rangle  \right|\left| \left\langle Bx,x \right\rangle  \right|.  
\end{aligned}\]
Therefore,
\[{{\left\| A+B \right\|}^{2}}\le {{\omega }^{2}}\left( A+iB \right)+2\left\| A \right\|\left\| B \right\|.\]
On the other hand, noting Proposition \ref{5},
\[\begin{aligned}
   {{\omega }^{2}}\left( A+iB \right)+2\left\| A \right\|\left\| B \right\| & \le \left\| {{ A }^{2}}+{{ B }^{2}} \right\|+2\left\| A \right\|\left\| B \right\| \\ 
 & \le {{\left\| A \right\|}^{2}}+{{\left\| B \right\|}^{2}}+2\left\| A \right\|\left\| B \right\| \\ 
 & \le {{\left( \left\| A \right\|+\left\| B \right\| \right)}^{2}}.  
\end{aligned}\]
Therefore,
\[\left\| A+B \right\|\le \sqrt{{{\omega }^{2}}\left( A+iB \right)+2\left\| A \right\|\left\| B \right\|}\le \left\| A \right\|+\left\| B \right\|.\]
This completes the proof. 
\end{proof}

{\tiny \vskip 0.3 true cm }

{\tiny (H. R. Moradi) Department of Mathematics, Payame Noor University (PNU), P.O. Box 19395-4697, Tehran, Iran.}

{\tiny \textit{E-mail address:} hrmoradi@mshdiau.ac.ir }

{\tiny \vskip 0.3 true cm }

{\tiny (M. Sababheh) Department of Basic Sciences, Princess Sumaya University For Technology, Al Jubaiha, Amman 11941, Jordan.}

{\tiny \textit{E-mail address:} sababheh@psut.edu.jo}

\end{document}